\documentclass[12pt]{article}
\usepackage{amsmath,amssymb,amsthm}
\usepackage{hyperref}
\usepackage{tikz}

\def\picA{\begin{tikzpicture}[scale=1]
\node[thick,circle,draw,inner sep=2pt] (a)  at (-1,0) {};
\node[thick,circle,draw,inner sep=2pt] (b)  at ( 0,0) {};
\node[thick,circle,draw,inner sep=2pt] (c)  at ( 1,0) {};
\node[thick,circle,draw,inner sep=2pt] (d)  at (-1,1) {};
\node[thick,circle,draw,inner sep=2pt] (e)  at ( 0,1) {};
\node[thick,circle,draw,inner sep=2pt] (f)  at ( 1,1) {};
\draw[thick] (a)--(b)--(c)--(f)--(e)--(d)--(a)  (b)--(e);

\node at (0,-0.6) {$\displaystyle p_G(x)=\frac{16}{63}x+\frac{197}{315}x^2+\frac{38}{315}x^3$};
\end{tikzpicture}
\hfil
\begin{tikzpicture}[scale=1]
\node[thick,circle,draw,inner sep=2pt] (a)  at (-1,0) {};
\node[thick,circle,draw,inner sep=2pt] (b)  at (-0.4,0.5) {};
\node[thick,circle,draw,inner sep=2pt] (c)  at (-1,1) {};
\node[thick,circle,draw,inner sep=2pt] (d)  at ( 1,1) {};
\node[thick,circle,draw,inner sep=2pt] (e)  at ( 0.4,0.5) {};
\node[thick,circle,draw,inner sep=2pt] (f)  at ( 1,0) {};
\draw[thick] (a)--(b)--(c)--(a) (f)--(e)--(d)--(f)  (b)--(e);

\node at (0,-0.6) {$\displaystyle p_G(x)=\frac{11}{63}x+\frac{247}{315}x^2+\frac{13}{315}x^3$};
\end{tikzpicture}}

\def\picC{
\begin{tikzpicture}[scale=1]
\node[thick,circle,draw,inner sep=2pt] (a)  at (-1,0) {};
\node[thick,circle,draw,inner sep=2pt] (b)  at (0,0) {};
\node[thick,circle,draw,inner sep=2pt] (c)  at (1,0) {};
\node[thick,circle,draw,inner sep=2pt] (d)  at (-0.5,-0.866) {};
\draw[thick] (b)--(d)--(a)--(b)--(c);
\node at (0,-1.436) {$\displaystyle p_G(x)=\frac{5}{6}x+\frac{1}{6}x^2$};
\end{tikzpicture}
\hfil
\begin{tikzpicture}[scale=1]
\node[thick,circle,draw,inner sep=2pt] (a)  at (-1,0) {};
\node[thick,circle,draw,inner sep=2pt] (b)  at (0,0) {};
\node[thick,circle,draw,inner sep=2pt] (c)  at (1,0) {};
\node[thick,circle,draw,inner sep=2pt] (d)  at (-0.5,-0.866) {};
\node[thick,circle,draw,inner sep=2pt] (e)  at (0.5,-0.866) {};
\draw[thick] (d)--(a)--(b)--(c)--(e);
\draw[thick,dashed] (d)--(e);
\node at (0,-1.436) {$\displaystyle p_G(x)=\frac{1}{3}x+\frac{2}{3}x^2$};
\end{tikzpicture}
\hfil
\begin{tikzpicture}[scale=1]
\node[thick,circle,draw,inner sep=2pt] (a)  at (-1,0) {};
\node[thick,circle,draw,inner sep=2pt] (b)  at (0,0) {};
\node[thick,circle,draw,inner sep=2pt] (c)  at (1,0) {};
\node[thick,circle,draw,inner sep=2pt] (d)  at (-0.5,-0.866) {};
\node[thick,circle,draw,inner sep=2pt] (e)  at (0.5,-0.866) {};
\draw[thick] (b)--(d)--(a)--(b)--(c)--(e);
\node at (0,-1.436) {$\displaystyle p_G(x)=\frac{5}{12}x+\frac{7}{12}x^2$};
\end{tikzpicture}

\bigskip

\begin{tikzpicture}[scale=1]
\node[thick,circle,draw,inner sep=2pt] (a)  at (-1,0) {};
\node[thick,circle,draw,inner sep=2pt] (b)  at (0,0) {};
\node[thick,circle,draw,inner sep=2pt] (c)  at (1,0) {};
\node[thick,circle,draw,inner sep=2pt] (d)  at (-0.5,-0.866) {};
\node[thick,circle,draw,inner sep=2pt] (e)  at (0.5,-0.866) {};
\draw[thick] (b)--(a)--(d)--(e)--(c)--(b)--(e);
\node at (0,-1.436) {$\displaystyle p_G(x)=\frac{7}{15}x+\frac{8}{15}x^2$};
\end{tikzpicture}
\hfil
\begin{tikzpicture}[scale=1]
\node[thick,circle,draw,inner sep=2pt] (a)  at (-1,0) {};
\node[thick,circle,draw,inner sep=2pt] (b)  at (0,0) {};
\node[thick,circle,draw,inner sep=2pt] (c)  at (1,0) {};
\node[thick,circle,draw,inner sep=2pt] (d)  at (-0.5,-0.866) {};
\node[thick,circle,draw,inner sep=2pt] (e)  at (0.5,-0.866) {};
\draw[thick] (a)--(b)--(c)--(e)--(d)--(c) (e)--(a)--(d);
\node at (0,-1.436) {$\displaystyle p_G(x)=\frac{1}{2}x+\frac{1}{2}x^2$};
\end{tikzpicture}
\hfil
\begin{tikzpicture}[scale=1]
\node[thick,circle,draw,inner sep=2pt] (a)  at (-1,0) {};
\node[thick,circle,draw,inner sep=2pt] (b)  at (0,0) {};
\node[thick,circle,draw,inner sep=2pt] (c)  at (1,0) {};
\node[thick,circle,draw,inner sep=2pt] (d)  at (-0.5,-0.866) {};
\node[thick,circle,draw,inner sep=2pt] (e)  at (0.5,-0.866) {};
\draw[thick] (a)--(b)--(c)--(e)--(b)--(d)--(a)--(e)  (c)--(d);
\node at (0,-1.436) {$\displaystyle p_G(x)=\frac{8}{15}x+\frac{7}{15}x^2$};
\end{tikzpicture}

\bigskip

\begin{tikzpicture}[scale=1]
\node[thick,circle,draw,inner sep=2pt] (a)  at (-1,0) {};
\node[thick,circle,draw,inner sep=2pt] (b)  at (0,0) {};
\node[thick,circle,draw,inner sep=2pt] (c)  at (1,0) {};
\node[thick,circle,draw,inner sep=2pt] (d)  at (-0.5,-0.866) {};
\node[thick,circle,draw,inner sep=2pt] (e)  at (0.5,-0.866) {};
\draw[thick] (a)--(b)--(c)--(e)--(b)--(d)--(a);
\node at (0,-1.436) {$\displaystyle p_G(x)=\frac{8}{15}x+\frac{7}{15}x^2$};
\end{tikzpicture}
\hfil
\begin{tikzpicture}[scale=1]
\node[thick,circle,draw,inner sep=2pt] (a)  at (-1,0) {};
\node[thick,circle,draw,inner sep=2pt] (b)  at (0,0) {};
\node[thick,circle,draw,inner sep=2pt] (c)  at (1,0) {};
\node[thick,circle,draw,inner sep=2pt] (d)  at (-0.5,-0.866) {};
\node[thick,circle,draw,inner sep=2pt] (e)  at (0.5,-0.866) {};
\draw[thick] (d)--(a)--(e) (a)--(b)--(c);
\node at (0,-1.436) {$\displaystyle p_G(x)=\frac{7}{12}x+\frac{5}{12}x^2$};
\end{tikzpicture}
\hfil
\begin{tikzpicture}[scale=1]
\node[thick,circle,draw,inner sep=2pt] (a)  at (-1,0) {};
\node[thick,circle,draw,inner sep=2pt] (b)  at (0,0) {};
\node[thick,circle,draw,inner sep=2pt] (c)  at (1,0) {};
\node[thick,circle,draw,inner sep=2pt] (d)  at (-0.5,-0.866) {};
\node[thick,circle,draw,inner sep=2pt] (e)  at (0.5,-0.866) {};
\draw[thick] (a)--(b)--(c)--(e)--(b)--(d)--(e)--(a)--(d);
\node at (0,-1.436) {$\displaystyle p_G(x)=\frac{62}{105}x+\frac{43}{105}x^2$};
\end{tikzpicture}

\bigskip

\begin{tikzpicture}[scale=1]
\node[thick,circle,draw,inner sep=2pt] (a)  at (-1,0) {};
\node[thick,circle,draw,inner sep=2pt] (b)  at (0,0) {};
\node[thick,circle,draw,inner sep=2pt] (c)  at (1,0) {};
\node[thick,circle,draw,inner sep=2pt] (d)  at (-0.5,-0.866) {};
\node[thick,circle,draw,inner sep=2pt] (e)  at (0.5,-0.866) {};
\draw[thick] (c)--(b)--(a)--(d)--(e) (b)--(d);
\node at (0,-1.436) {$\displaystyle p_G(x)=\frac{19}{30}x+\frac{11}{30}x^2$};
\end{tikzpicture}
\hfil
\begin{tikzpicture}[scale=1]
\node[thick,circle,draw,inner sep=2pt] (a)  at (-1,0) {};
\node[thick,circle,draw,inner sep=2pt] (b)  at (0,0) {};
\node[thick,circle,draw,inner sep=2pt] (c)  at (1,0) {};
\node[thick,circle,draw,inner sep=2pt] (d)  at (-0.5,-0.866) {};
\node[thick,circle,draw,inner sep=2pt] (e)  at (0.5,-0.866) {};
\draw[thick] (c)--(b)--(a)--(d)--(e)--(b)--(d) (a)--(e);
\node at (0,-1.436) {$\displaystyle p_G(x)=\frac{13}{20}x+\frac{7}{20}x^2$};
\end{tikzpicture}
\hfil
\begin{tikzpicture}[scale=1]
\node[thick,circle,draw,inner sep=2pt] (a)  at (-1,0) {};
\node[thick,circle,draw,inner sep=2pt] (b)  at (0,0) {};
\node[thick,circle,draw,inner sep=2pt] (c)  at (1,0) {};
\node[thick,circle,draw,inner sep=2pt] (d)  at (-0.5,-0.866) {};
\node[thick,circle,draw,inner sep=2pt] (e)  at (0.5,-0.866) {};
\draw[thick] (a)--(d)--(b)--(e)--(c)--(d)--(e)--(a);
\node at (0,-1.436) {$\displaystyle p_G(x)=\frac{23}{35}x+\frac{12}{35}x^2$};
\end{tikzpicture}

\bigskip

\begin{tikzpicture}[scale=1]
\node[thick,circle,draw,inner sep=2pt] (a)  at (-1,0) {};
\node[thick,circle,draw,inner sep=2pt] (b)  at (0,0) {};
\node[thick,circle,draw,inner sep=2pt] (c)  at (1,0) {};
\node[thick,circle,draw,inner sep=2pt] (d)  at (-0.5,-0.866) {};
\node[thick,circle,draw,inner sep=2pt] (e)  at (0.5,-0.866) {};
\draw[thick] (d)--(a)--(b)--(c)--(d)--(b)--(e);
\node at (0,-1.436) {$\displaystyle p_G(x)=\frac{41}{60}x+\frac{19}{60}x^2$};
\end{tikzpicture}
\hfil
\begin{tikzpicture}[scale=1]
\node[thick,circle,draw,inner sep=2pt] (a)  at (-1,0) {};
\node[thick,circle,draw,inner sep=2pt] (b)  at (0,0) {};
\node[thick,circle,draw,inner sep=2pt] (c)  at (1,0) {};
\node[thick,circle,draw,inner sep=2pt] (d)  at (-0.5,-0.866) {};
\node[thick,circle,draw,inner sep=2pt] (e)  at (0.5,-0.866) {};
\draw[thick] (a)--(b)--(d)--(e)--(b)--(c);
\node at (0,-1.436) {$\displaystyle p_G(x)=\frac{23}{30}x+\frac{7}{30}x^2$};
\end{tikzpicture}
\hfil
\begin{tikzpicture}[scale=1]
\node[thick,circle,draw,inner sep=2pt] (a)  at (-1,0) {};
\node[thick,circle,draw,inner sep=2pt] (b)  at (0,0) {};
\node[thick,circle,draw,inner sep=2pt] (c)  at (1,0) {};
\node[thick,circle,draw,inner sep=2pt] (d)  at (-0.5,-0.866) {};
\node[thick,circle,draw,inner sep=2pt] (e)  at (0.5,-0.866) {};
\draw[thick] (c)--(e)--(d)--(a)--(b)--(d) (b)--(e);
\draw[thick,dashed] (b)--(c);
\node at (0,-1.436) {$\displaystyle p_G(x)=\frac{17}{30}x+\frac{13}{30}x^2$};
\end{tikzpicture}
}

\newcommand{\picDshort}[2]{
\begin{tikzpicture}[scale=1]
\node[thick,circle,draw,inner sep=2pt] (a)  at (-1.5,0) {};
\node[thick,circle,draw,inner sep=2pt] (b)  at (-0.5,0) {};
\node[thick,circle,draw,inner sep=2pt] (c)  at ( 0.5,0) {};
\node[thick,circle,draw,inner sep=2pt] (d)  at ( 1.5,0) {};
\node[thick,circle,draw,inner sep=2pt] (e)  at (-1,-0.866) {};
\node[thick,circle,draw,inner sep=2pt] (f)  at ( 0,-0.866) {};
\node[thick,circle,draw,inner sep=2pt] (g)  at ( 1,-0.866) {};
\draw[thick] #1;
\node at (0,-1.686) {\small $ \displaystyle p_G(x)=#2$};
\end{tikzpicture}}

\def\picD{
\begin{tikzpicture}[scale=1]
\node[thick,circle,draw,inner sep=2pt] (a)  at (-1.5,0) {};
\node[thick,circle,draw,inner sep=2pt] (b)  at (-0.5,0) {};
\node[thick,circle,draw,inner sep=2pt] (c)  at ( 0.5,0) {};
\node[thick,circle,draw,inner sep=2pt] (d)  at ( 1.5,0) {};
\draw[thick] (a)--(b)--(c)--(d);
\node at (0,-1.686) {\small $\displaystyle  p_G(x)=\frac{2}{3}x+\frac{1}{3}x^2$};
\end{tikzpicture} \hfil
\begin{tikzpicture}[scale=1]
\node[thick,circle,draw,inner sep=2pt] (a)  at (-1.5,0) {};
\node[thick,circle,draw,inner sep=2pt] (b)  at (-0.5,0) {};
\node[thick,circle,draw,inner sep=2pt] (c)  at ( 0.5,0) {};
\node[thick,circle,draw,inner sep=2pt] (d)  at ( 1.5,0) {};
\node[thick,circle,draw,inner sep=2pt] (e)  at (-1,-0.866) {};
\draw[thick] (e)--(a)--(b)--(c)--(d);
\node at (0,-1.686) {\small $\displaystyle  p_G(x)=\frac{1}{3}x+\frac{2}{3}x^2$};
\end{tikzpicture}\hfil
\begin{tikzpicture}[scale=1]
\node[thick,circle,draw,inner sep=2pt] (a)  at (-1.5,0) {};
\node[thick,circle,draw,inner sep=2pt] (b)  at (-0.5,0) {};
\node[thick,circle,draw,inner sep=2pt] (c)  at ( 0.5,0) {};
\node[thick,circle,draw,inner sep=2pt] (d)  at ( 1.5,0) {};
\node[thick,circle,draw,inner sep=2pt] (e)  at (-1,-0.866) {};
\draw[thick] (a)--(b)--(c)--(d) (b)--(e);
\node at (0,-1.686) {\small $\displaystyle  p_G(x)=\frac{7}{12}x+\frac{5}{12}x^2$};
\end{tikzpicture}

\bigskip

\begin{tikzpicture}[scale=1]
\node[thick,circle,draw,inner sep=2pt] (a)  at (-1.5,0) {};
\node[thick,circle,draw,inner sep=2pt] (b)  at (-0.5,0) {};
\node[thick,circle,draw,inner sep=2pt] (c)  at ( 0.5,0) {};
\node[thick,circle,draw,inner sep=2pt] (d)  at ( 1.5,0) {};
\node[thick,circle,draw,inner sep=2pt] (e)  at (-1,-0.866) {};
\node[thick,circle,draw,inner sep=2pt] (g)  at ( 1,-0.866) {};
\draw[thick] (e)--(a)--(b)--(c)--(d)--(g);
\node at (0,-1.686) {\small $\displaystyle  p_G(x)=\frac{2}{15}x+\frac{11}{15}x^2+\frac{2}{15}x^3$};
\end{tikzpicture}\hfil
\begin{tikzpicture}[scale=1]
\node[thick,circle,draw,inner sep=2pt] (a)  at (-1.5,0) {};
\node[thick,circle,draw,inner sep=2pt] (b)  at (-0.5,0) {};
\node[thick,circle,draw,inner sep=2pt] (c)  at ( 0.5,0) {};
\node[thick,circle,draw,inner sep=2pt] (d)  at ( 1.5,0) {};
\node[thick,circle,draw,inner sep=2pt] (e)  at (-1,-0.866) {};
\node[thick,circle,draw,inner sep=2pt] (g)  at ( 1,-0.866) {};
\draw[thick] (a)--(b) (e)--(b)--(c)--(d)--(g);
\node at (0,-1.686) {\small $\displaystyle  p_G(x)=\frac{1}{4}x+\frac{3}{4}x^2$};
\end{tikzpicture}\hfil 
\begin{tikzpicture}[scale=1]
\node[thick,circle,draw,inner sep=2pt] (a)  at (-1.5,0) {};
\node[thick,circle,draw,inner sep=2pt] (b)  at (-0.5,0) {};
\node[thick,circle,draw,inner sep=2pt] (c)  at ( 0.5,0) {};
\node[thick,circle,draw,inner sep=2pt] (d)  at ( 1.5,0) {};
\node[thick,circle,draw,inner sep=2pt] (f)  at ( 0,-0.866) {};
\node[thick,circle,draw,inner sep=2pt] (g)  at ( 1,-0.866) {};
\draw[thick] (a)--(b)--(f)--(c)--(d) (f)--(g);
\node at (0,-1.686) {\small $\displaystyle  p_G(x)=\frac{3}{10}x+\frac{17}{30}x^2+\frac{2}{15}x^3$};
\end{tikzpicture}

\bigskip

\begin{tikzpicture}[scale=1]
\node[thick,circle,draw,inner sep=2pt] (a)  at (-1.5,0) {};
\node[thick,circle,draw,inner sep=2pt] (b)  at (-0.5,0) {};
\node[thick,circle,draw,inner sep=2pt] (c)  at ( 0.5,0) {};
\node[thick,circle,draw,inner sep=2pt] (d)  at ( 1.5,0) {};
\node[thick,circle,draw,inner sep=2pt] (e)  at (-1,-0.866) {};
\node[thick,circle,draw,inner sep=2pt] (g)  at ( 1,-0.866) {};
\draw[thick] (a)--(b)--(c)--(d)  (e)--(b)  (c)--(g);
\node at (0,-1.686) {\small $\displaystyle  p_G(x)=\frac{7}{15}x+\frac{8}{15}x^2$};
\end{tikzpicture}\hfil
\begin{tikzpicture}[scale=1]
\node[thick,circle,draw,inner sep=2pt] (a)  at (-1.5,0) {};
\node[thick,circle,draw,inner sep=2pt] (b)  at (-0.5,0) {};
\node[thick,circle,draw,inner sep=2pt] (c)  at ( 0.5,0) {};
\node[thick,circle,draw,inner sep=2pt] (d)  at ( 1.5,0) {};
\node[thick,circle,draw,inner sep=2pt] (e)  at (-1,-0.866) {};
\node[thick,circle,draw,inner sep=2pt] (f)  at ( 0,-0.866) {};
\draw[thick] (a)--(b)--(c)--(d)  (e)--(b)  (f)--(b);
\node at (0,-1.686) {\small $\displaystyle  p_G(x)=\frac{11}{20}x+\frac{9}{20}x^2$};
\end{tikzpicture}\hfil
\picDshort{(b)--(a)--(e)--(f)--(g)--(d)--(c)}{\frac{2}{45}x+\frac{26}{45}x^2+\frac{17}{45}x^3}

\bigskip

\picDshort{(a)--(b)  (e)--(b)--(c)--(d)--(g)--(f)}{\frac{31}{360}x+\frac{67}{90}x^2+\frac{61}{360}x^3}\hfil
\picDshort{(e)--(a)--(b)--(f)  (b)--(c)--(g)--(d)}{\frac{41}{360}x+\frac{109}{180}x^2+\frac{101}{360}x^3}\hfil
\picDshort{(a)--(b)--(c)--(f)--(e)  (c)--(d)--(g)}{\frac{3}{20}x+\frac{9}{20}x^2+\frac{2}{5}x^3}

\bigskip

\picDshort{(a)--(b)--(f)--(c)--(d)  (e)--(b)  (c)--(g)}{\frac{1}{6}x+\frac{5}{6}x^2}\hfil
\picDshort{(a)--(b)--(f)--(g)--(d)  (e)--(b)--(c)}{\frac{13}{60}x+\frac{47}{60}x^2}\hfil
\picDshort{(a)--(b)--(f)--(c)--(d) (c)--(g)  (e)--(f)}{\frac{79}{360}x+\frac{11}{18}x^2+\frac{61}{360}x^3}

\bigskip

\picDshort{(a)--(b)--(f)--(c)--(d)  (e)--(f)--(g)}{\frac{17}{60}x+\frac{8}{15}x^2+\frac{11}{60}x^3}\hfil
\picDshort{(a)--(b)--(e)  (f)--(b)--(c)--(d)  (c)--(g)}{\frac{5}{12}x+\frac{7}{12}x^2}\hfil
\picDshort{(a)--(b)--(c)--(d)   (e)--(c)--(f)  (c)--(g)}{\frac{8}{15}x+\frac{7}{15}x^2}
}

\def\picE{\begin{tikzpicture}[scale=1,rotate=90]
\node[thick,circle,draw,inner sep=2pt] (a1) at (0,0) {};
\node[thick,circle,draw,inner sep=2pt] (a2) at (0,1) {};
\node[thick,circle,draw,inner sep=2pt] (a3) at (0,2) {};
\node[thick,circle,draw,inner sep=2pt] (a4) at (0,3) {};

\node[thick,circle,draw,inner sep=2pt] (b1) at (-1,0.5) {};
\node[thick,circle,draw,inner sep=2pt] (b2) at (-1,1.5) {};
\node[thick,circle,draw,inner sep=2pt] (b3) at (-1,2.5) {};

\node[thick,circle,draw,inner sep=2pt] (c1) at (1,1) {};
\node[thick,circle,draw,inner sep=2pt] (c2) at (1,2) {};

\draw[thick] (a1)--(b1)--(a2)--(b2)--(a1)--(b3)--(a2)  (a3)--(b1)--(a4)--(b2)--(a3)--(b3)--(a4)  (c1)--(a1)--(c2)--(a2)--(c1)--(a3)--(c2)--(a4)--(c1)--(c2);

\end{tikzpicture} \hfil
\begin{tikzpicture}[scale=1,rotate=90]
\node[thick,circle,draw,inner sep=2pt] (a1) at (0,0) {};
\node[thick,circle,draw,inner sep=2pt] (a2) at (0,1) {};
\node[thick,circle,draw,inner sep=2pt] (a3) at (0,2) {};
\node[thick,circle,draw,inner sep=2pt] (a4) at (0,3) {};

\node[thick,circle,draw,inner sep=2pt] (b1) at (-1,0.5) {};
\node[thick,circle,draw,inner sep=2pt] (b2) at (-1,1.5) {};
\node[thick,circle,draw,inner sep=2pt] (b3) at (-1,2.5) {};

\node[thick,circle,draw,inner sep=2pt] (c1) at (1,1) {};
\node[thick,circle,draw,inner sep=2pt] (c2) at (1,2) {};

\draw[thick] (a1)--(b1)--(a2)--(b2)--(a1)--(b3)--(a2)  (a3)--(b1)--(a4)--(b2)--(a3)--(b3)--(a4)  (c1)--(a1)--(c2)--(a2)--(c1)--(a3)--(c2)--(a4)--(c1);

\end{tikzpicture}}

\allowdisplaybreaks

\theoremstyle{plain}
\newtheorem{theorem}{Theorem}

\newtheorem{corollary}{Corollary}

\newtheorem{conjecture}{Conjecture}
\theoremstyle{definition}
\newtheorem{definition}{Definition}

\DeclareMathOperator{\vol}{vol}

\title{A forest building process on simple graphs}

\author{Zhanar Berikkyzy\thanks{Dept.\ of Mathematics, Iowa State University, Ames, IA, USA.\newline {\tt \{zhanarb,butler\}@iastate.edu}} \and
Steve Butler\footnotemark[1] \and
Jay Cummings\thanks{Dept.\ of Mathematics and Statistics, Sacramento State University, Sacramento, CA, USA. \newline{\tt J.Cummings19@gmail.com}}\and
Kristin Heysse\thanks{Department of Mathematics, Statistics, and Computer Science, Macalester College, St. Paul, MN USA.\ \newline{\tt kheysse@macalester.edu}} \and
Paul Horn\thanks{Dept.\ of Mathematics, University of Denver, Denver, CO, USA. {\tt paul.horn@du.edu}}\and
Ruth Luo\thanks{Dept.\ of Mathematics, University of Illinois at Urbana-Champaign, Champaign, IL, USA. \newline{\tt ruthluo2@illinois.edu}}
\and
Brent Moran\thanks{Berlin Mathematical School, Freie Universit\"at Berlin, Germany. {\tt blm@zedat.fu-berlin.de}}}

\date{\empty}
\begin{document}
\maketitle

\begin{abstract}
Consider the following process on a simple graph without isolated vertices:  Order the edges randomly and keep an edge if and only if it contains a vertex which is not contained in some preceding edge.  The resulting set of edges forms a spanning forest of the graph.

The probability of obtaining $k$ components in this process for complete bipartite graphs is determined as well as a formula for the expected number of components in any graph.  A generic recurrence and some additional basic properties are discussed.
\end{abstract}

\section{Introduction}
Given a simple graph $G$ with no isolated vertices, we consider the following \emph{forest building} process to form a subgraph of $G$:
\begin{quote}
Consider $S$, the empty graph on $V(G)$, and some ordering of the edges of $G$, $e_1,e_2,\ldots,e_m$. For every edge, add the edge $e_j$ to $S$ if $e_j$ is incident to a vertex not incident to $e_i$ for all $i<j$.
\end{quote}
We note that this process can be thought of as starting with the empty graph and then considering edges one at a time in the ordering and adding to the graph only those edges which connect to at least one isolated vertex.

The resulting subgraph must span $V(G)$ because every vertex is incident to some edge, and so the first time we consider an incident edge we will keep that edge.  Further, we cannot form cycles since we would need to keep an edge both of whose vertices have previously been seen.  The result of the process is a spanning forest of $G$ without isolated vertices, and moreover any spanning forest without isolated vertices will occur for \emph{some} ordering of the edges.

This process was implicitly considered in the edge flipping problem (see \cite{BCCG, CG}).  In particular, let $F(G,k)$ denote the number of edge orderings in the forest building process so that the resulting forest has $k$ components, let $P(G,k)$ denote the probability that a randomly, uniformly chosen forest building process (i.e., a random ordering of the edges) produces a graph with $k$ components (note $m!P(G,k)=F(G,k)$), and let $p_G(x)=\sum_kP(G,k)x^k$ be the generating function of the $P(G,k)$ terms.  Then the following was shown in \cite{BCCG}.

\begin{theorem}[Butler-Chung-Cummings-Graham]\label{thm:complete}
For the complete graph $K_n$ we have 
\[
p_{K_n}(x) = \sum_k\frac{{n-1\choose n-2k,k,k-1}2^{n-2k}}{{2n-2\choose n}}x^k.
\]
\end{theorem}
(Throughout the paper we follow the convention that a sum with no bounds is interpreted to run over all values which have nonzero terms.)

We will establish the corresponding result for complete bipartite graphs in Section~\ref{sec:Kmn}.  We show how to determine the expected number of connected components of the forest building process in Section~\ref{sec:expected} for an arbitrary graph.  We give a recurrence relationship for the polynomial $p_G(x)$ in terms of polynomials for some of the subgraphs in Section~\ref{sec:recurrence} and use this to establish some basic results.   Finally, we give some concluding remarks in Section~\ref{sec:conclusion}.

\section{Complete bipartite graphs}\label{sec:Kmn}
We now consider the analogue of Theorem~\ref{thm:complete} for complete bipartite graphs $K_{s,t}$. 

\begin{theorem}\label{thm:bipartite}
For the complete bipartite graph $K_{s,t}$, we have
\[
p_{K_{s,t}}(x)=\sum_k\frac{k(s+t){s\choose k}{t\choose k}}{st{s+t \choose s}}x^k.
\]
\end{theorem}

\begin{proof}
We need to show that $P(K_{s,t},k)={k(s+t){s\choose k}{t\choose k}}/({st{s+t \choose s}})$, and to do this we will find it useful to work on a recurrence which incorporates intermediate states in the process on $K_{s,t}$.  In particular, for $s$ and $t$ fixed let $Q_{s,t}(a,b,\ell)$ be the probability that the process will finish with $k$ components \emph{given} that we have already used $s-a$ vertices in the part with size $s$ and $t-b$ vertices in the part with size $t$ and currently have $k-\ell$ components.

From this definition we immediately recover some initial conditions as follows:
\begin{equation}\label{eq:init}
Q_{s,t}(a,0,\ell)=Q_{s,t}(0,b,\ell)=\begin{cases}
1 & \text{if }\ell=0,\\
0 & \text{if }\ell\neq0,
\end{cases}
\end{equation}
since we cannot introduce any new components (i.e., we have already saturated one side of the complete bipartite graph and so every new edge will hook onto an existing tree).  Hence the current number of components will no longer change. If we have the required $k$ components ($\ell=0$), then we must succeed, and otherwise we must fail.  We also adopt the convention that $Q_{s,t}(a,b,-1)=0$ for all $a$ and $b$ (i.e., we have too many components and we cannot reduce the number of components by adding new edges).

We also note that there is a simple recurrence that must be satisfied.  This is because if we are in a given situation corresponding to a given choice of $a,b,\ell$ then there are precisely three things that can happen which affect the state of the current process.
\begin{itemize}
\item We add an edge incident to a new vertex in the part of size $s$ and an old vertex in the part of size $t$.  There are $a(t-b)$ such edges and then we are in a situation corresponding to $a-1,b,\ell$.
\item We add a new edge incident to an old vertex in the part of size $s$ and a new vertex in the part of size $t$.  There are $(s-a)b$ such edges and then we are in a situation corresponding to $a,b-1,\ell$.
\item We add a new edge incident to new vertices in both parts.  There are $ab$ such edges and then we are in a situation corresponding to $a-1,b-1,\ell-1$.
\end{itemize}
Altogether there are $st-(s-a)(t-b)=at+bs-ab$ such edges.  Putting this together gives the following recurrence:
\begin{multline}\label{eq:recur}
Q_{s,t}(a,b,\ell)=
\frac{a(t-b)}{at+bs-ab}Q_{s,t}(a-1,b,\ell)+
\frac{(s-a)b}{at+bs-ab}Q_{s,t}(a,b-1,\ell)\\+
\frac{ab}{at+bs-ab}Q_{s,t}(a-1,b-1,\ell-1).
\end{multline}

Solving such a recurrence is nontrivial, but verifying a recurrence is straightforward.  In this case, the solution to the recurrence is
\[
Q_{s,t}(a,b,\ell)=\frac{{b\choose \ell}{s+t-b-1\choose a-\ell}}{{s+t-1\choose a}}=\frac{{a\choose \ell}{s+t-a-1\choose b-\ell}}{{s+t-1\choose b}}.
\]
We will omit the verification that this satisfies \eqref{eq:init} and \eqref{eq:recur}, as this involves standard computations.

Finally, we have
\[
P(K_{s,t},k)=Q_{s,t}(s,t,k)=\frac{{t\choose k}{s-1\choose s-k}}{{s+t-1\choose s}}=\frac{k(s+t){s\choose k}{t\choose k}}{st{s+t \choose s}}.\qedhere
\]
\end{proof}

\subsection{Complete multipartite graphs}
For the proof of Theorem~\ref{thm:bipartite}, we implicitly relied on the high degree of symmetry in the graph. Specifically, at every stage we have some variant of a complete bipartite graph to work with.  As a result, the basic approach we used for complete bipartite graphs would also work for complete multipartite graphs and would give similar initial conditions as in \eqref{eq:init} and similar recurrence as in \eqref{eq:recur}.

The problem lies in solving the recurrence.  For the complete bipartite graph, the solution to the recurrence was found by inspection of small cases, where it was noticed that the resulting probabilities consisted of products of small prime factors.  For tripartite graphs and larger, this is no longer the case and small cases can involve large prime factors.  So we have no natural candidate for a general solution.  As an example, for the graph $K_{3,3,3}$ we have
\[
p_{K_{3,3,3}}(x)=
 \frac{ 1992}{26125}x
+\frac{11724}{26125}x^2
+\frac{10951}{26125}x^3
+\frac{ 1458}{26125}x^4,
\]
and we note $11724=2^2{\cdot}3{\cdot}977$ for its prime factorization.  It is an open problem to find simple closed form solutions for complete multipartite graphs.

\section{Expected number of components}\label{sec:expected}
For a given graph $G$ the expected number of components in the forest building process can be determined by considering $\sum_k kP(G,k)$.  However this involves first determining $P(G,k)$ which can be difficult.  There is an easier way to determine the number of expected components, as shown in the next result.

\begin{theorem}\label{thm:numcomp}
Given a graph $G$ without isolated vertices, the expected number of connected components in the forest building process is
\begin{equation}\label{eq:expected}
\sum_{uv\in E(G)}\frac1{d(u)+d(v)-1},
\end{equation}
where $d(u)$ and $d(v)$ indicates the degree of $u$ and $v$, respectively.
\end{theorem}
\begin{proof}
Let us consider the forest building process.  When we come to a particular edge one of three things occurs.  Namely, we have already seen both vertices, in which case we discard the edge; we have only seen one vertex, in which case we add the edge onto an already existing tree in our forest; or we have not seen either vertex, in which case we start a new tree in our forest.  To determine the number of connected components, we have to consider the number of times an edge is added and neither vertex has been seen before.

For the edge $uv\in E(G)$, there are $d(u)+d(v)-1$ edges which contain either vertex $u$ or vertex $v$.  The probability that in a random ordering we will see $uv$ first, i.e., this edge adds a new component, is $1/(d(u)+d(v)-1)$.  Therefore, by linearity of expectation, the number of expected components is what is given in \eqref{eq:expected}.
\end{proof}

We note that in general just knowing the degrees of the vertices on each edge is not enough to determine the $P(G,k)$.  As a simple example, consider the pair of graphs in Figure~\ref{fig:pair}.  For the two graphs the edges collectively have the same corresponding degrees of vertices, and hence we have that the expected number of components is $28/15$, but the individual probabilities are distinct.

\begin{figure}
\centering

\picA

\caption{Two graphs with equal edge-degree sequences and different polynomials $P_G(x)$}
\label{fig:pair}
\end{figure}

We can apply Theorem~\ref{thm:numcomp} to the complete graphs and complete bipartite graphs to establish the following.

\begin{corollary}
For the complete graph $K_n$ with $n\ge 2$, the number of expected components is $n(n-1)/(4n-6)$.  For the complete bipartite graph $K_{s,t}$, the number of expected components is $st/(s+t-1)$.
\end{corollary}
\begin{proof}
For $K_n$ we have for each edge that $d(u)+d(v)-1=2n-3$ and that there are ${n\choose 2}$ such edges.  Now apply Theorem~\ref{thm:numcomp}.

For $K_{s,t}$ we have for each edges that $d(u)+d(v)-1=s+t-1$ and that there are $st$ such edges.  Now apply Theorem~\ref{thm:numcomp}.
\end{proof}

The results of the preceding corollary can also be proven directly without the aid of Theorem~\ref{thm:numcomp} by means of combinatorial identities.  For example for the complete graph we need to show that
\[
\sum_k kP(K_n,k)=\sum_k 
\frac{k2^{n-2k}{n-1\choose n-2k,k,k-1}}{{2n-2\choose n}}
=\frac{n(n-1)}{4n-6}.
\]
To start this we begin by noting,
\begin{equation}\label{eq:combident}
\sum_K2^{N-2K}{N\choose K}{N-K\choose N-2K}={2N\choose N}.
\end{equation}
The right hand side counts the number of $N$-element sets from $\{1,2,\ldots,2N\}$.  We now show that the left hand side counts the same thing.  First put the elements into $N$ pairs, i.e., $\{1,N+1\}$, $\{2,N+2\}$, \dots, $\{N,2N\}$, we now choose $K$ of these pairs (in ${N\choose K}$ ways) and take both elements from these pairs; from the remaining $N-K$ pairs we choose $N-2K$ of these pairs (in ${N-K\choose N-2K}$ ways) and from each of these pairs we take one element ($2^{N-2K}$ ways) to form our set with $N$ elements.  As $K$ runs over all possibilities we will form all $N$ element subsets giving the result.  Now setting $N=n-2$ and $K=k-1$ then \eqref{eq:combident} becomes
\[
\sum_k2^{n-2k}{n-2\choose k-1}{n-k-1\choose n-2k}={2n-4\choose n-2}=\frac{n}{4n-6}{2n-2\choose n}.
\]
Noting that ${n-2\choose k-1}{n-k-1\choose n-2k}=\frac{k}{n-1}{n-1\choose n-2k,k,k-1}$ and rearranging the terms then gives the desired result.

A similar, and simpler, argument works for complete bipartite graphs.

Looking at the results for the complete graph, we have that the number of connected components in the forest building process tends to $n/4$.  In general we can apply Theorem~\ref{thm:numcomp} to conclude that for any $d$-regular graph, where $d=d(n)$ tends to infinity with $n$, the expected number of connected components in the forest building process also tends to $n/4$.  Similarly, for $(p,q)$-biregular graphs (i.e., bipartite graphs with parts of sizes $s,t$ where the degrees in one part are all $p$ and the degrees in the other part are all $q$), Theorem~\ref{thm:numcomp} can be used to show if $p=p(n)$ and $q=q(n)$ tend to infinity with $n$, the expected number of connected components in the forest building process tends to $st/(s+t)$.

\subsection{Components tend to emerge quickly}

As we go through the process for a dense graph we should expect that many edges at the start are initially used and most edges at the end will be discarded.  In particular we should expect that the final components will emerge quickly.  We can make this more precise for complete graphs.

For an edge $e$ in the graph we let $d'(e)$ denote the number of edges which are incident to $e$, equivalently if $e=uv$ then $d'(e)=d(u)+d(v)-2$.  Let $B$ be a uniformly, randomly chosen ordering of the edges of $G$, and let $\kappa(G,B)$ denote the number of connected components in the forest building process of $G$ with respect to edge ordering $B$.  Theorem~\ref{thm:numcomp} can now be restated as
\[
\mathbb{E} (\kappa(G,B)) = \sum_{e \in E(G)} \frac{1}{d'(e) + 1}.
\]
For graphs where $d'(e)$ is constant this simplifies nicely.  For general graphs we can get a bound on this expression, as done in the following corollary.

\begin{corollary}\label{cor:bound}
For any graph $G$ with $n$ vertices and $m$ edges, let $B$ be a random ordering of the edges of $G$.    Then
\[
\mathbb{E} (\kappa(G,B)) \geq \frac{ m }{\mathbb{E}(d') + 1}.
\]
Let $H$ be a randomly chosen graph $n$ vertices and $m$ edges (i.e., $H\in G(n,m)$).  Then
\[
\mathbb{E} \big(\kappa(H,B)\big) \geq \frac{mn + m}{4m + n - 3}.
\]
\end{corollary}

\begin{proof}
Using linearity of expectation and an application of Jensen's inequality we have the following:
\begin{multline*}
\mathbb{E}(\kappa(G, B))
=\mathbb{E}(\mathbb{E}(\kappa(G, B))) 
= \mathbb{E}  \bigg(\sum_{e \in E(G)} \frac{1}{d'(e) + 1}\bigg) 
\\= \sum_{e \in E(G)} \mathbb{E} \bigg( \frac{1}{d'(e) + 1} \bigg) 
\geq \frac{ m }{ \mathbb{E}(d') + 1 }.
\end{multline*}
Now consider a graph $H\in G(n,m)$.  For an edge $e=uv$ there are $n-2$ possible edges incident to $u$ and $n-2$ possible edges incident to $v$.  The probability that any one such edge exists in the graph will be $(m-1)/({n\choose 2}-1)$.  So we have that $\mathbb{E}(d') = 2(n-2)(m-1)/({n\choose 2}-1)$, putting this in and simplifying gives the result.
\end{proof}

We now apply this to the complete graph by noting that if we pause the forest building process after $m$ edges then this is equivalent to looking at the forest building process on a graph in $G(n,m)$.  By Corollary~\ref{cor:bound} this indicates that the expected number of edges in the process is $(mn+m)/(4m+n-3)$.  When $m\gg n$ then this is approximately $n/4$, while on the other hand the final number of expected components is also approximately $n/4$.  We can conclude that once we have seen a superlinear number of edges that we will typically have already formed most of the components.  (In terms of random graphs, this says that random graphs with super linear, i.e., $\omega(n)$, number of edges have very few, i.e., $o(n)$, isolated vertices.)

With a bit more work we can determine the exact expression for the expected number of components in a graph in $G(n,m)$.

\begin{corollary}\label{cor:randomexp}
Let $G\in G(n, m)$ denote a random graph with $n$ vertices and $m$ edges.  Then
\[
  \mathbb{E} \left(\kappa (G, B) \right) =
    \frac{{n \choose 2}}{(2n - 3)}
    \bigg(
      1 - \frac{ {{n \choose 2} - m \choose 2n - 3} }
      { {{n \choose 2} \choose 2n - 3} }
    \bigg).
\]
\end{corollary}

\begin{proof}
Consider a fixed edge $e$.  For a random graph in $G (n, m)$, we have that $d'(e)$ is hypergeometric with the following parameters.
\begin{itemize}
  \item Population size:  $N := {n \choose 2} - 1$.
  \item Number of edges (``successes'') in population:  $K := m - 1$.
  \item Number of trials:  $t := 2 (n - 2)$.
\end{itemize}
This allows us to conclude
\[
  \mathbb{P} \left( 
    d' (e) = d
  \right) = \frac{ {K \choose d} {N - K \choose t-d} }{ {N \choose t} },
\]
so we now calculate (starting in the second line $\widehat{{*}} = {*}+1$)
\begin{align*}
  \mathbb{E}\left( \frac{1}{d'(e) +1} \right)
  &= \sum_{d=0}^t \left( \frac{1}{d+1} \right)
  \frac{ {K \choose d} {N - K \choose t - d} }{ {N \choose t} } 
  = \sum_{d=0}^t \left( \frac{1}{K+1} \right)
  \frac{ {K + 1 \choose d + 1} {N - K \choose t - d} }{ {N \choose t} } \\
  &= \frac{1}{\widehat{K}} \sum_{\widehat{d}=1}^{\widehat{t}}
  \frac{
    {\widehat{K} \choose \widehat{d}}
    {\widehat{N} - \widehat{K} \choose \widehat{t} - \widehat{d}}
  }{ {\widehat{N} - 1 \choose \widehat{t} - 1} }
  = \frac{1}{\widehat{K}} \sum_{\widehat{d}=1}^{\widehat{t}}
  \frac{
    {\widehat{K} \choose \widehat{d}}
    {\widehat{N} - \widehat{K} \choose \widehat{t} - \widehat{d}}
  }{ \frac{\widehat{t}}{\widehat{N}}{\widehat{N} \choose \widehat{t}} } 
  \\&= \frac{\widehat{N}}{\widehat{t}\widehat{K}} \sum_{\widehat{d}=1}^{\widehat{t}}
  \frac{
    {\widehat{K} \choose \widehat{d}}
    {\widehat{N} - \widehat{K} \choose \widehat{t} - \widehat{d}}
  }{ {\widehat{N} \choose \widehat{t}} } 
  = \frac{\widehat{N}}{\widehat{t}\widehat{K}} \bigg( \underbrace{\bigg(
    \sum_{\widehat{d}=0}^{\widehat{t}}
    \frac{
      {\widehat{K} \choose \widehat{d}}
      {\widehat{N} - \widehat{K} \choose \widehat{t} - \widehat{d}}
    }{ {\widehat{N} \choose \widehat{t}} }
  \bigg)}_{=1}
  - \frac{ {\widehat{N} - \widehat{K} \choose \widehat{t}} }
  { {\widehat{N} \choose \widehat{t}} }
  \bigg)
  \\&= \frac{\widehat{N}}{\widehat{t}\widehat{K}} \bigg( 1
  - \frac{ {\widehat{N} - \widehat{K} \choose \widehat{t}} }
  { {\widehat{N} \choose \widehat{t}} }
  \bigg) 
  = \frac{{n \choose 2}}{(2n - 3) m}
  \bigg(
    1 - \frac{ {{n \choose 2} - m \choose 2n - 3} }
    { {{n \choose 2} \choose 2n - 3} }
  \bigg).
\end{align*}

Now, we have that
\[
\mathbb{E}(\kappa(G, B))
= \sum_{e \in E( G )} \mathbb{E} \bigg( \frac{1}{d'(e) + 1} \bigg)
= \sum_{e \in E( G )} \frac{{n \choose 2}}{(2n - 3) m}
  \bigg(
    1 - \frac{ {{n \choose 2} - m \choose 2n - 3} }
    { {{n \choose 2} \choose 2n - 3} }
  \bigg)
\]
and since the particular edge $e$ plays no role in the sum, we conclude
\[  \mathbb{E}(\kappa(G, B))
    = \frac{{n \choose 2}}{(2n - 3)}
  \bigg(
    1 - \frac{ {{n \choose 2} - m \choose 2n - 3} }
    { {{n \choose 2} \choose 2n - 3} }
  \bigg).\qedhere
\]
\end{proof}

\section{Recurrence relationship}\label{sec:recurrence}
Computing $p_G(x)$ by computing the probabilities directly is difficult because there are $|E(G)|!$ possible orderings to consider.  However, we can simplify the process by giving a recurrence relationship that relates $p_G(x)$ to some of the polynomials  $p_H(x)$ where the $H$ are subgraphs of $G$.  The key is to group the edge orderings by the \emph{last} edge considered.  Note that if the last edge is incident to a leaf, then it will be included in the graph and otherwise it will not be included (because both the incident vertices would have occurred earlier).  We will assume that the graph $G$ has no isolated edges.

If we make the convention that $p_{K_1}(x)=1$, or equivalently that isolated vertices don't contribute to the component count (hence the last edge doesn't effect component count), then we have the following:
\begin{align*}
|E(G)|!p_G(x)&=\sum_kF(G,k)x^k\\
&=\sum_{e\in E(G)}\bigg(\sum_kF(G-e,k)x^k\bigg)\\
&=\sum_{e\in E(G)}\bigg((|E(G)|-1)!\sum_kP(G-e,k)x^k\bigg)\\
&=(|E(G)|-1)!\sum_{e\in E(G)}p_{G-e}(x).
\end{align*}

We can summarize this in the following result which can be used to efficiently compute $p_G(x)$ recursively.

\begin{theorem}\label{thm:recurrence}
Let $p_{K_1}(x)=1$, $p_{K_2}(x)=x$.  Then for the disjoint union of graphs $G$ and $H$, denoted $G\,\mathaccent\cdot\cup\, H$ we have
\[
p_{G\mathaccent\cdot\cup H}(x)=p_G(x)p_H(x),
\]
and for any graph $G$ without isolated edges we have
\[
p_G(x)=\frac{1}{|E(G)|}\sum_{e\in E(G)}p_{G-e}(x).
\]
\end{theorem}
\begin{proof}
The only thing that remains is the disjoint union result.  For this we note that when the graph has disjoint components then the edges in one component have no effect in another component, i.e., the edge revealing process can be run independently in each component.  The result now follows by noting 
\[
P(G\,\mathaccent\cdot\cup \,H,k)=\sum_\ell P(G,\ell)P(H,k-\ell).\qedhere
\]
\end{proof}

\subsection{Edge-transitive graphs}
In one special case of Theorem~\ref{thm:recurrence} the recurrence simplifies tremendously at the first stage.  Namely for edge-transitive graphs, i.e., graphs where for any two edges $e_1$ and $e_2$ there is an automorphism of the graph sending $e_1$ to $e_2$.  We note that both complete graphs and complete bipartite graphs belong to this family.

\begin{corollary}\label{cor:edgetrans}
Let $G$ be an edge-transitive graph without isolated edges, and let $H=G-e'$ where $e'$ is some edge of $G$.  Then $p_G(x)=p_H(x)$.
\end{corollary}
\begin{proof}
Fix $e'\in E(G)$ where $G$ is an edge-transitive graph.  Then by Theorem~\ref{thm:recurrence} we have
\begin{multline*}
p_G(x)=\frac{1}{|E(G)|}\sum_{e\in E(G)}p_{G-e}(x)=
\frac{1}{|E(G)|}\sum_{e\in E(G)}p_{G-e'}(x)\\
=\frac{1}{|E(G)|}|E(G)|p_H(x)=p_H(x),
\end{multline*}
where the second equality follows by the transitivity of edges, i.e., removal of any edge produces the same graph.
\end{proof}

This can be used to construct pairs of graphs $G$ and $H$ with $p_G(x)=p_H(x)$ by letting $G$ be an edge transitive graph and $H$ the graph obtained by deleting one edge from $G$.  Examples include complete graphs, complete bipartite graphs, cycles and so on.

There are other examples of pairs of connected graphs on the same number of vertices with the same polynomials, but they seem rare.  For example up through nine vertices there are fifteen such pairs not explained by Corollary~\ref{cor:edgetrans}.  Most of these appear erratic, but there is one construction which explains six of these pairs.  Namely let $G_{2k+1}=K_{k,k+1}+e$ where $e$ is an edge connecting two vertices in the larger part (see Figure~\ref{fig:G9} for $G_9$).  Then for $k=2,3,4$ we have that $G_{2k+1}$ and $K_{k,k+1}$ have the same polynomial, the other three pairs are $G_{2k+1}$ and $K_{k,k+1}$ with some edge removed.  (Note that for $k=1$ we have that $G_3=K_3$ and $K_{1,2}=P_3$ which we know have the same polynomials by Corollary~\ref{cor:edgetrans}.)

\begin{figure}[htb]
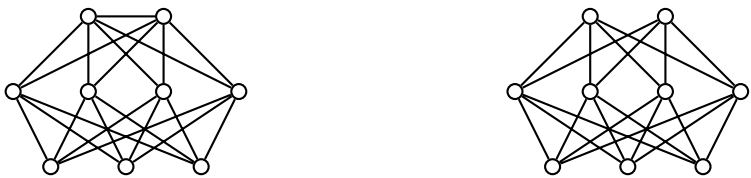

\centering

\picE

\caption{The graphs $G_9$ and $K_{4,5}$}
\label{fig:G9}
\end{figure}

\begin{conjecture}
For $k\ge 1$ the graph $G_{2k+1}=K_{k,k+1}+e$, where $e$ is an edge connecting two vertices in the larger part, has the same polynomial as $K_{k,k+1}$.
\end{conjecture}

\subsection{Paths}
The polynomial $p_G(x)$ implicitly relies on understanding the polynomial for all connected subgraphs of $G$.  Since the set of non-isomorphic connected subgraphs tends to be large, this makes it difficult to determine $p_G(x)$.  There are however some exceptions, the simplest being paths which only has paths as connected subgraphs.  Let $P_n$ denote the path on $n$ vertices and let $f_n(x)=p_{P_{n+1}}(x)$ under the convention of Theorem~\ref{thm:recurrence}, i.e., $f_0(x)=1$, $f_1(x)=x$, and so on.

The following result on generating functions for paths is similar to the one proved by Chung and Graham \cite{CG}.

\begin{theorem}\label{thm:paths}
Let $Q(t)=\sum_{n\ge 0}f_n(x)t^n$.  Then for $x>1$,
\[
Q(t)=\sqrt{x-1}\tan\bigg(t\sqrt{x-1}+\arctan\bigg(\frac1{\sqrt{x-1}}\bigg)\bigg).
\]
\end{theorem}
\begin{proof}
By Theorem~\ref{thm:recurrence} we have for $n\ge 2$
\[
nf_n(x)=f_0(x)f_{n-1}(x)+f_1(x)f_{n-2}(x)+\cdots+f_{n-1}(x)f_0(x).
\]
In terms of the generating function $Q(t)$ the left hand side corresponds to the coefficients of $Q'(t)$ and the right hand side corresponds to the coefficients of $(Q(t))^2$ and then we only need to correct for the constant term.  In particular we have
\[
Q'(t)=\big(Q(t)\big)^2+(x-1)\qquad\text{with }Q(0)=1.
\]
This is a differential equation and it can be readily verified that the solution to this differential equation is the function
\[
Q(t)=\sqrt{x-1}\tan\bigg(t\sqrt{x-1}+\arctan\bigg(\frac1{\sqrt{x-1}}\bigg)\bigg).\qedhere
\]
\end{proof}

\subsection{One component}
The probability that the graph will have one component at the end of the process is the coefficient of $x$ in the polynomial $p_G(x)$. We can use the recurrence from Theorem~\ref{thm:recurrence} for this special case.  The key observation to make is that if $e$ is a \emph{large bridge}, namely an edge whose removal leaves two components with at least one edge in each component, then the coefficient of $x$ in $p_{G-e}(x)$ is $0$ (i.e., we cannot have one component).  On the other hand if we are not a large bridge then having $e$ occur last will have no impact on whether the final process will have a single component.  So if we let $B(G)$ be the set of large bridges we have
\[
P(G,1) = \frac{1}{|E(G)|}\sum_{e\in E(G)\setminus B(G)}P(G-e,1).
\]

For stars it is easy to see that $P(K_{1,s},1)=1$, but for most graphs we would expect that the probability of having one component is extremely small.  As an example for the cycle $C_n$ the probability of one component is $n2^{n-2}/n!$ which is super-exponentially small in $n$.

This behavior though is in some sense an outlier and we establish exponential bounds for having a single component for a large family of graphs (including almost all ``random'' graphs).  To begin we will need the notion of the Cheeger constant.

\begin{definition}
The \emph{Cheeger constant} of a graph $G$, denoted $\Phi(G)$, is 
\begin{equation}\label{eq:cheeger}
\Phi(G) = \min_{\substack{ X \subseteq V(G) \\ \vol(X) \leq \frac{1}{2} \vol(V)}} \frac{|E(X,V\setminus X)|}{\vol(X)},
\end{equation}
where $E(X,Y)$ is the set of edges joining $X$ and $Y$ and $\vol(X)=\sum_{v\in X}d(v)$.
\end{definition}

The Cheeger constant is a way to measure the efficiency of cutting a graph into two components and so is a measure of connectivity.  It is known that if the Cheeger constant for a family of graphs is bounded away from zero then the graphs behave quasi-randomly (for a thorough treatment of the Cheeger constant see \cite{fan}).

\begin{theorem}\label{THEOREM}
Fix $\epsilon > 0$.  Then there exist real numbers $c = c(\epsilon)>0$ and $C<1$, so that if $G$ is a $d$-regular graph with $\Phi(G) \geq \epsilon$, then
\[
c^{n} < P(G,1) < C^{n}.  
\]
\end{theorem}

We remark here that the exponential upper bound does not use connectivity in any essential way, but the lower bound is sensitive to connectivity.  Indeed, one can construct $d$-regular graphs which have super exponentially small probabilities of having one component by taking a necklace of graphs which have two vertices of degree $d-1$ and the remaining vertices of degree $d$ (such graphs have low connectivity).  On the other hand there is no construction of a large regular graph which has high probability of having one component.

\begin{proof} 
We randomly generate an ordering of the edges, $e_1, e_2, \ldots, e_{nd/2}$.  Let $1=t_1 < t_2 < t_3 <\cdots$ denote the times in which a new vertex is seen by an edge (since the first edge sees two new vertices there will be at most $n-1$ times a new vertex is seen).  It suffices to estimate the probability that $e_{t_i}$ is incident to a vertex already present.  Let $\mathcal{A}_i$ denote this event.  Then we want to bound 
\[
\mathbb{P}(\bigcap_i \mathcal{A}_i) = \prod_{i = 1}^{n-1} \mathbb{P}(\mathcal{A}_{i} | \bigcap_{j < i} \mathcal{A}_j),  
\]
and so it suffices to estimate 
\[
\mathbb{P}(\mathcal{A}_{i} | \bigcap_{j < i} \mathcal{A}_j).  
\]

To lower bound this, consider any set when the $t_i$-th edge was added.  At this point, the first $t_i-1$ edges are incident to a collection of $i$ (connected) vertices.  We break into two cases depending on whether $i \leq n/2$ or $i > n/2$.   

{\bf Case 1:} $i \leq \frac{n}{2}$

Let $X$ denote the vertices incident to the first $t_i-1$ edges.  We are interested in the probability that the next chosen edge lies in $E(X,V\setminus X)$ given that it either joins a vertex in $X$ to a vertex in $V\setminus X$ or has both incident vertices in $V\setminus X$. We have $\vol(V\setminus X) = (n-i)d$, and the number of edges induced in $V\setminus X$ is $\frac{1}{2} \big((n-i)d - |E(X, V\setminus X)|\big)$.   Thus the probability we are interested in is
\begin{equation}
\frac{|E(X,V\setminus X)|}{\frac{1}{2} \big((n-i)d - |E(X,V\setminus X)|\big) + |E(X,V\setminus X)|} = 
\frac{2|E(X,V\setminus X)|}{(n-i)d + |E(X,V\setminus X)|}. \label{eqn1}
\end{equation}
This is an increasing function of $|E(X,V\setminus X)|$ and by \eqref{eq:cheeger} we have $|E(X,V\setminus X)|\geq di \Phi $.  Also, trivially, $|E(X,V\setminus X)| < di$.

Thus 
\[
\frac{2i}{n}  =  \frac{2di}{(n-i)d + id} > \eqref{eqn1} \geq \frac{2di\Phi }{(n-i)d + di\Phi } = \frac{2i\Phi}{n + (\Phi-1) i}.
\]

{\bf Case 2:} $i > \frac{n}{2}$

This can be bounded similarly, getting a bound of 
\[
1 > \mathbb{P}(\mathcal{A}_i | \bigcap_{j < i} \mathcal{A}_j) > \frac{2di\Phi}{ id + \Phi id} = \frac{2\Phi}{1 + \Phi}.  
\]

Thus a total lower bound on the probability, taking the product all the way up to $n$, is given by 
\begin{align*}
&\bigg( \prod_{i=1}^{n/2} \frac{2\Phi i}{n + (\Phi-1)i} \bigg) \bigg( \prod_{i=n/2+1}^n \frac{2\Phi}{1+\Phi}\bigg) \\
&= \bigg(\frac{2\Phi^2}{1 + \Phi}\bigg)^{n/2} \frac{(n/2)!2^{n/2} }{n^{n/2}} \prod_{i=1}^{n/2} \bigg( \frac{1}{1 + (\Phi-1)i/n}\bigg) \\
&> \bigg(\frac{2\Phi^2}{1 + \Phi}\bigg)^{n/2} \frac{(n/2)!2^{n/2} }{n^{n/2}} \exp\bigg((1-\Phi) \sum_{i=1}^{n/2} \frac{i}{n}  \bigg) \\
&= (1+o(1)) \sqrt{\pi n}  \bigg(\frac{2\Phi^2}{1 + \Phi}\bigg)^{n/2} e^{-n/2} \exp\bigg((1-\Phi) \frac{n}{4}  \bigg)
\end{align*}
which is of the desired form.  For the upper bound we have
\[
\prod_{i=1}^{n/2}\frac{2i}{n}
=\frac{(n/2)!2^{n/2}}{n^{n/2}}=(1+o(1))\sqrt{\pi n}e^{-n/2}
\]
which is also of the desired form.
\end{proof}

As a corollary of this, we observe that almost all $d$-regular graphs on $n$ vertices have this property.  Indeed, Bollob\'as \cite{B} showed that for fixed $d$, the Cheeger constant of a random $d$-regular graph has $\Phi(G) \geq (1 - \eta(d))/2$ with probability $1-o(1)$, where $\eta(d)$ is a number satisfying $2^{4/d} < (1-\eta)^{1-\eta}(1+\eta)^{1+\eta}$.  This, combined with the proof of Theorem~\ref{THEOREM} gives an exponential  lower bound on $P(G,1)$.

For graphs where $d$ is not constant, random $d$-regular graphs are harder to deal with directly.  In particular, a direct analogue of Bollob\'as's result for $d =d(n)$ for some growing function of $n$ is unknown.  However, an analogue to understand $\Phi(G)$ for such graphs is given to us by using graph spectra.  Building on earlier work of Friedman, Kahn and Szemer\'edi \cite{FKS}, Broder, Frieze, Suen and Upfal \cite{BFSU} showed that if $d=o(n^{1/2})$, and $G$ is a random $d$-regular graph, then the second adjacency eigenvalue of $G$ is $\lambda=O(\sqrt{d})$.  Combining this with the Cheeger inequality (see, e.g., Chung \cite{fan})
\[
\frac12\Phi^2<\frac{d-\lambda}{d}\le 2\Phi,
\]
one observes that for random $d$-regular graphs where $d=d(n)$ is growing but $o(n^{1/2})$, then $\Phi(G) \ge \frac{1}{2} - o(1)$, and applying the results of the above theorem again yields an exponential bound. 

\section{Concluding remarks}\label{sec:conclusion}
We have considered a forest building process on simple graphs without isolated vertices and have given the probability of ending up with $k$ trees for complete graphs, complete bipartite graphs, or one of these graphs with a single edge deleted.  For completeness, we give $p_G(x)$ for all graphs on at most five vertices not covered in the preceding results in the Appendix as well as all trees up through seven vertices.

Most of our focus has been on understanding the graph polynomial $p_G(x)$.  The study of graph polynomials has a long and rich history (see \cite{Merino1, Merino2}), which provides ample material for additional exploration.  We note that this $p_G(x)$ is not a simple variant of the Tutte polynomial as there exists graphs on five vertices with the same Tutte polynomial but where the corresponding $p_G(x)$ disagree.

There are still many things that are not known about the forest building process.  For example: Which graph operations work well with the polynomials?  What are some families where $p_G(x)$ can be explicitly computed?  Do the numbers $P(G,k)$ form a log-concave sequence for all graphs $G$?  (This last question has been confirmed for all graphs through $9$ vertices.) Do there exist two trees, $S$ and $T$, on $n$ vertices with $p_S(x)=p_T(x)$?  (Through $n=19$ the answer is no; requiring that they be on the same number of vertices is important because all stars have the same polynomial.)

We hope to see some of these questions, and more, addressed in future work.

\bigskip

\noindent \textbf{Acknowledgements:}~~This research was started at the Graduate Research Workshop in Combinatorics (GRWC) held at Iowa State University in June 2015 and which was supported by the grant NSF DMS 1500662.  The authors thank the anonymous referee for a thorough reading and useful suggestions of the text.

\eject
\section*{Appendix -- $p_G(x)$ for small graphs}

\begin{figure}[htb]
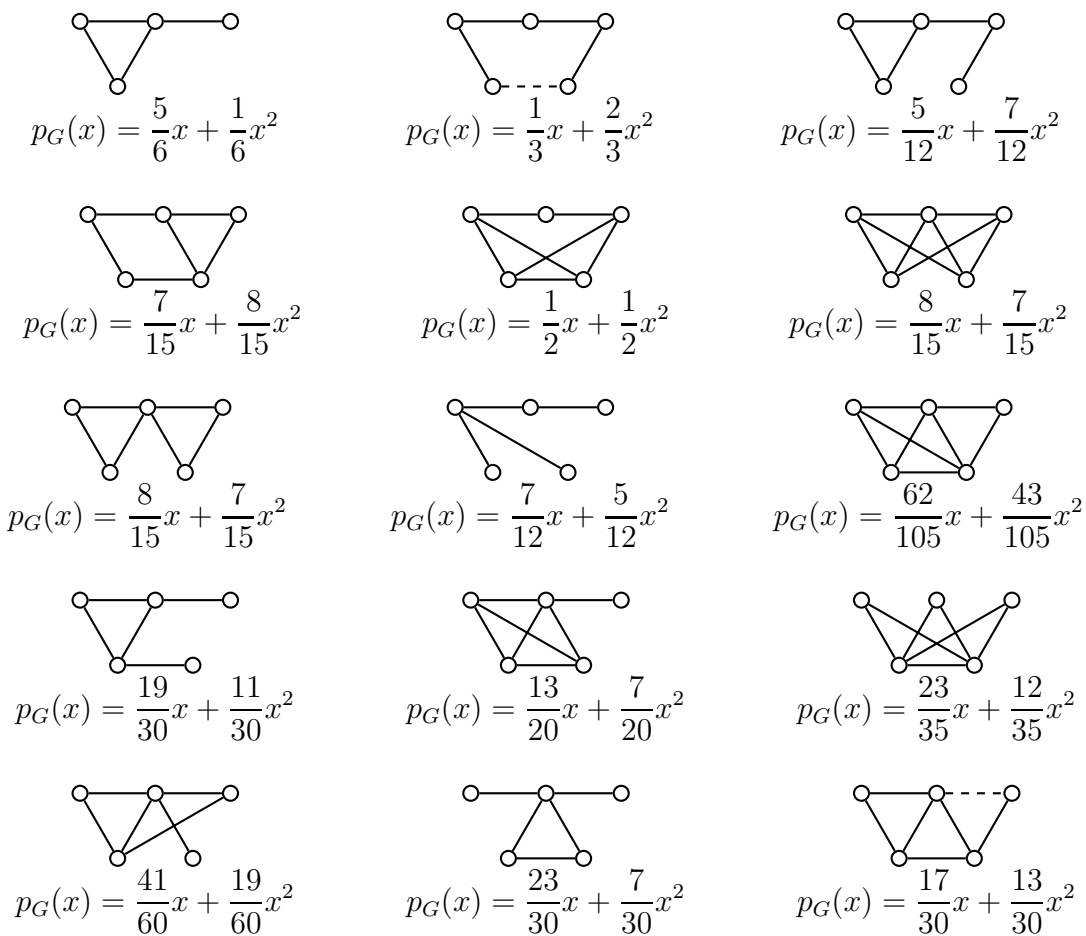

\centering

\picC

\caption{$p_G(x)$ for some small graphs}
\end{figure}

\begin{figure}[htb]
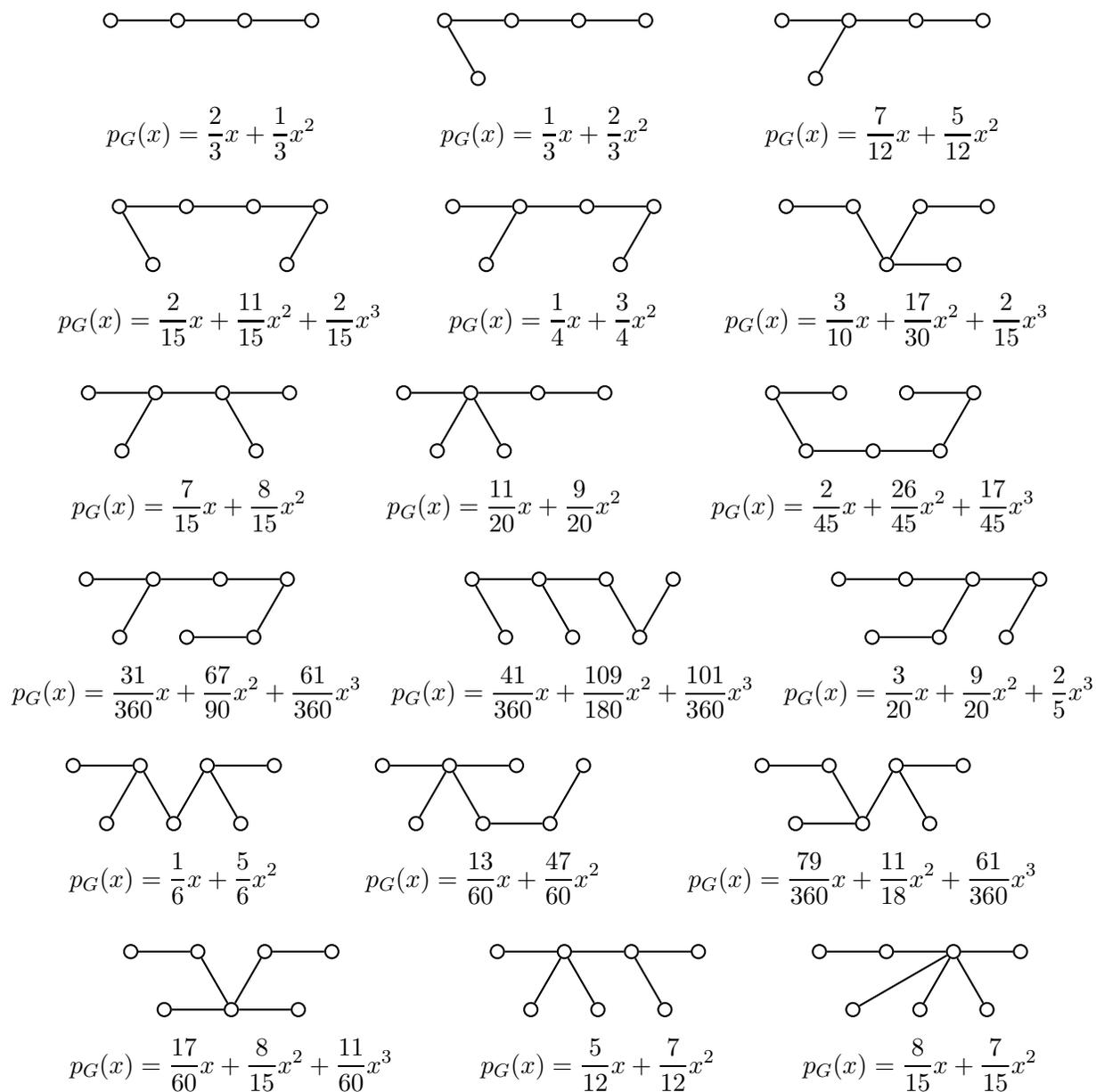

\centering

\picD

\caption{$p_G(x)$ for trees through seven vertices}
\end{figure}

\begin{thebibliography}{99}

\bibitem{B} B. Bollob\'{a}s, The isoperimetric number of random regular graphs, \textit{European J. Combin.}, \textbf{9} (1988) pp. 241--244.

\bibitem{BFSU}  A. Z. Broder, A. M. Frieze, S. Suen, and E. Upfal, Optimal construction of edge disjoint paths in random graphs, \textit{SIAM J. Comput.}, \textbf{28} (1999) pp. 541--573.

 \bibitem{BCCG}
S. Butler, F. Chung, J. Cummings and R. Graham, Edge flipping in the complete graph, \textit{Adv. in Applied Math.} \textbf{69} (2015), 46--64.

\bibitem{fan}
F. Chung, \textit{Spectral Graph Theory, 2nd edition}, AMS Publications, 1997.

\bibitem{CG}
F. Chung and R. Graham, Edge flipping in graphs, \textit{Adv. in Applied Math.} \textbf{48} (2012), 37--63.

\bibitem{Merino1} J. A. Ellis-Monaghan and C. Merino, Graph polynomials and their applications I: The Tutte polynomial, in \textit{Structural analysis of complex networks}, edited by M. Dehmer, Spriner, New York, 2011, 219--255.

\bibitem{Merino2} J. A. Ellis-Monaghan and C. Merino, Graph polynomials and their applications II: Interrelations and interpretations, in \textit{Structural analysis of complex networks}, edited by M. Dehmer, Spriner, New York, 2011, 257--292.

\bibitem{FKS} J. Friedman, J. Kahn, and E. Szemer\'edi, On the second eigenvalue of random regular graphs, \textit{Proceedings of the 21st Annual ACM Symposium on Theory of Computing} (1989), pp. 587--598.

\end{thebibliography}
\end{document}